\documentclass[12pt,reqno]{amsart} 

\usepackage{mathdots}
\usepackage{amsmath, amssymb} 
\usepackage{amscd}            
\usepackage{amsxtra}          
\usepackage{upref}            

\usepackage[mathscr]{eucal}

\theoremstyle{plain}

\newtheorem{theorem}{Theorem}
\newtheorem{proposition}[theorem]{Proposition}

\newtheorem{corollary}[theorem]{Corollary}
\newtheorem{lemma}[theorem]{Lemma}

\newtheorem{definition}[theorem]{Definition}

\theoremstyle{remark}
\newtheorem{remark}[theorem]{Remark}

\numberwithin{equation}{section}
\numberwithin{theorem}{section}
\newcommand{\be}%
  {\protect\setcounter{equation}{\value{subsubsection}}}  
\newcommand{\ee}%
  {\protect\setcounter{subsubsection}{\value{equation}}}

\newcommand{\Z}{\mathbb{Z}}
\newcommand{\Q}{\mathbb{Q}}

\newcommand{\C}{\mathbb{C}} 
\newcommand{\I}{\mathcal{I}}

\newcommand{\aq}{\mathbb{A}_{\Q}}
\newcommand{\af}{\mathbb{A}_{F}}

\newcommand{\gltwo}{{\rm GL}_2}
\newcommand{\glr}{{\rm GL}_r}

\newcommand{\gln}{{\rm GL}_n}

\newcommand{\glnfv}{\gln (F_v)}
\newcommand{\glrfv}{\glr (F_v)}
\newcommand{\gltwonfv}{{\rm GL}_{2n}(F_v)}
\newcommand{\gltwonaf}{{\rm GL}_{2n}(\af)}

\newcommand{\gltwoaq}{\gltwo(\aq)}

\newcommand{\cs}{\mathcal S}

\newcommand{\wpiv}{{\mathcal W}(\pi_v,\psi_v)}

\newcommand{\res}{\text{Re}(s)}

\newcommand{\ljsepiv}{L_{JS}(s,\pi_v,\wedge^2)}
\newcommand{\lgepiv}{L(s,\pi_v,\wedge^2)}

\newcommand{\jswvphiv}{J(s,W_v,\phi_v)}
\newcommand{\jswv}{J(s,W_v)}
\newcommand{\lshepiv}{L_{Sh}(s,\pi_v,\wedge^2)}
\newcommand{\ljsepi}{L_{JS}(s,\pi,\wedge^2)}
\newcommand{\lshepi}{L_{Sh}(s,\pi,\wedge^2)}
\newcommand{\lgepi}{L(s,\pi,\wedge^2)}

\DeclareMathOperator{\tr}{Tr}

\begin{document}


\title {On the local and global exterior square $L$-functions of $\text{GL}_n$}
\author{Pramod~Kumar~Kewat} 
\address{The Institute of Mathematical Sciences\\
         Chennai 600113\\ India}
\email{pramodkk@imsc.res.in}
\subjclass{11} 
\author{Ravi~Raghunathan} 
\address{Department of Mathematics \\ 
         Indian Institute of Technology Bombay\\
         Mumbai 400076\\ India}
\email{ravir@math.iitb.ac.in}

\begin{abstract} 
We show that the local exterior square $L$-functions of $\rm{GL}_n$ constructed via the 
theory of integral representations by Jacquet and Shalika coincide with 
those constructed by the Langlands-Shahidi method for square integrable
representations (and for all irreducible representations when $n$ is even).
We also deduce several local and global consequences.
\end{abstract}

\maketitle


\markboth{P.K. Kewat and R. Raghunathan}{On the exterior square $L$-function}

\section{Introduction}
Let $F$ be a number field, $v$ a place of $F$ and $F_v$ its completion. To 
any irreducible admissible representation $\pi_v$ of ${\rm GL}_{n}(F_v)$ 
the local Langlands correspondence attaches an $n$-dimensional representation 
$\rho_{F_v}(\pi_v)$ of the Weil-Deligne group, when $F_v$ is a $p$-adic field, and 
of the Weil group, when $F_v$ is archimedean. The exterior square 
$L$-function of $\pi_v$ is defined via this correspondence as a Galois 
$L$-function --
\[
L(s,\pi_v,\wedge^2):=L(s,\wedge^2(\rho_{F_v}(\pi_v)),
 \]
and there have been different approaches to establishing its analytic properties.
In \cite{JaSh90a}, Jacquet and Shalika suggested that 
$\ljsepiv$, defined as the ``greatest common divisor" of certain local 
integrals denoted $\jswvphiv$, when $n$ is even, and $\jswv$, when $n$ is odd, should
yield the local exterior square $L$-function, when $v$ is $p$-adic.
At the archimedean places, $\ljsepiv$ is defined via the local Langlands correspondence,
as before.

On the other hand, the Langlands-Shahidi method also provides a potential
construction for this $L$-function which we denote by $\lshepiv$, for any place
$v$. The corresponding global $L$-functions are defined by
\[
\ljsepi=\prod_v\ljsepiv\quad\text{and}\quad\lshepi=\prod_v\lshepiv.
\]
The main result of this paper is the following theorem.
\begin{theorem}\label{sqinteqthm} If $\pi_v$ is an irreducible, smooth, square integrable 
representation of
$\glnfv$, we have
\begin{equation}\label{sqinteq}
\ljsepiv=\lshepiv.
\end{equation}
\end{theorem}
When $n$ is even, it is possible to express the exterior square $L$-function of an arbitrary 
irreducible generic representation in terms of  $L$-functions of the inducing quasi-square 
integrable data (see \cite{CoPS942}). This allows us to prove the following theorem.
\begin{theorem}\label{loceqthm} If $\pi_v$ is an irreducible generic representation of 
$\gltwonfv$, we have
\begin{equation}\label{loceq}
\ljsepiv=\lshepiv.
\end{equation}
\end{theorem}
As an immediate corollary, we obtain the following global result.
\begin{theorem}\label{globeqthm}  Let $\pi=\otimes'_v\pi_v$ be a cuspidal 
automorphic representation of $\gltwonaf$, then
\begin{equation}\label{globeq}
\ljsepi=\lshepi.
\end{equation}
\end{theorem}
The analogous global theorem in the odd case is less satisfactory, but probably suffices for
many applications.

The results of Henniart in \cite{Henniart10} (see Theorem \ref{henniart} in Section 
\ref{lsmglobal}) show that the equality of the local and global
Langlands-Shahidi $L$-functions with the corresponding local and global Galois $L$-functions. 
This enables us to deduce the following corollary.
\begin{corollary}\label{jsgaleq}
Under the hypotheses of Theorems \ref{sqinteqthm}, \ref{loceqthm} and \ref{globeqthm}, we have
\[
\ljsepiv=\lgepiv\quad\text{and}\quad \ljsepi=\lgepi.
\]
\end{corollary}
Several other pleasant consequences, both local and global, follow from the equalities 
of $L$-functions established above. The analytic properties of $\lshepi$
-- entireness, the functional equation and boundedness in vertical strips -- have been 
established by Shahidi, Kim and Gelbart-Shahidi in a series of papers 
(\cite{Shahidi81,Shahidi90,Kim99,GeSh01}) in many cases. Theorem 
\ref{globeqthm} of this paper,
allows us to show that $\ljsepi$ has the same properties. Our proof of the 
holomorphy results is is different from the one due to Belt (see Theorem 5.2 of \cite{Belt11}),
who excludes the ramified places. The functional equation and boundedness in vertical strips
are new results for $\ljsepi$. On the other hand,
Belt's global theorem for the case when $n$ is even and $\pi$ self dual
allows us to deduce the corresponding analytic behaviour of $\lshepi$.
The analytic properties of the global exterior square $L$-functions are
discussed in Corollaries \ref{analyticjseven}, \ref{analyticjsodd} and 
\ref{analyticlsh} in Section \ref{extensions} of this paper. 

There are also consequences for the local $L$-functions. We are able to 
give a characterisation of self dual square integrable representations 
in terms of the existence of a pole for the local symmetric square $L$-function
in Corollary \ref{symsq}.
A comparison of the two global $L$-functions also enables us to push 
the theory of the local exterior square $L$-functions via integral representations a little further.  
It is now possible to define a local $\epsilon$-factor and establish a local functional 
equation for square integrable representations, in fact, for all generic representations
occurring as a local constituent of a cuspidal automorphic representation in 
the even case (see Theorem \ref{locfnaleqn}).
This is a somewhat indirect method of obtaining the functional equation -- indeed the local 
functional equation has no direct analogue in the Langlands-Shahidi constructions.
We hope to find a proof within the local theory of integral representations in the future
which will work for all irreducible representations.

Theorem \ref{sqinteq} is actually proved by global methods by using techniques
similar to those used (more recently) in \cite{AnRa05} and \cite{Henniart10},
but we need to work somewhat harder since the local theory is not as complete in 
our case. In particular, we note that although the integral representation yields the 
$L$-function at the unramfied places for a suitable choice of test vector, it is not known
whether this choice yields the generator of the relevant fractional ideal (that is, whether it is the 
``greatest common divisor" of the local integrals). Indeed, this last fact only follows 
after we have proved Theorem \ref{loceqthm} in the even case, and remains an problem
in the odd case.

The crucial new inputs are the recent holomorphy and non-vanishing
results of the first author (see Theorem M of \cite{Kewat11}) at the finite places, and 
a non-vanishing result of Dustin Belt (Theorem 2.2 of \cite{Belt11}) at the archimedean places. 
The idea is to embed the square integrable representation as the local component of a 
cuspidal automorphic representation. One then takes the quotient of the global
integral of Jacquet and Shalika by $\lshepi$ to obtain a quotient of finitely many 
local factors. For suitable choices of local data we can 
show that the quotient of the non-archimedean factors must be entire and non-vanishing.
This last argument is dependent on some slightly finer local analysis,
involving the local $\epsilon$-factors, which must be suitably defined in our context.
Once this is done, arguments involving the locations of the possible poles allow us to conclude 
that the relevant quotient is identically $1$, yielding the equality of the two $L$-functions.

The paper is organised as follows. Section \ref{pre} deals with notation and the preliminaries, while
Sections \ref{irm} and \ref{lsm} review what is already known from the theory of integral representations
and the Langlands-Shahidi method respectively. In Section \ref{esqfcr} we prove a lemma about 
certain $\epsilon$-factors (only identified as such later).
The proof of Theorem \ref{sqinteqthm} is treated in Section \ref{proof} and the main extensions and 
corollaries are given in Sections \ref{extensions} and \ref{locfneqn}.

We would like to thank Dustin Belt for very kindly making an early draft of his paper \cite{Belt11} 
available to us, and to Jim Cogdell for bringing Belt's non-vanishing results to our notice.
We would also like to thank U. K. Anandavardhanan for many helpful conversations and several useful
suggestions.

\section{Notation and Preliminaries} \label{pre}
Throughout this paper $F$ will be a number field, $v$ a place of $F$ and 
$F_v$ its completion at the place $v$. Let $\af$ denote
the ring of ad\`eles over $F$. Let $ |x_v|_v $ denote the absolute value of an 
element $x_v$ of $F_v$ and $q_v$ be the cardinality 
of the residue field of $F_v$. If $x$ is in $\af$ and the $x_v$ are its local
components, $|x|$ denotes the product $\prod_v|x_v|_v$
of the local absolute values. We denote by  $\mathcal{S}(F_v^{n})$ (resp. $\mathcal{S}(\af^n)$) 
the space of Schwartz-Bruhat functions on $F_v^{n} $ (resp. $\af^n$). For $\phi_v \in \mathcal{S}(F_v^{n})$ 
(resp. $\phi\in\mathcal{S}(\af^n)$), we denote by $\hat{\phi}_v$ (resp. $\hat{\phi}_v$) the Fourier transform 
of $\phi_v$ (resp. $\phi$).

We let ${\rm G}$ stand for the group $\gln$.
The $F$ points of ${\rm G}$ will be denoted ${\rm G}(F)$, its $F_v$ points by $G_v$ 
and its $\af$ points by $G$. We follow this convention,
whenever convenient, for all the algebraic 
groups defined over $F$ that arise in this paper.
We let ${\rm N}$ be the standard maximal unipotent subgroup, that is, the 
subgroup of ${\rm G}$ consisting 
of upper triangular matrices with $1$ in each diagonal entry, and
we let ${\rm Z}$ be the center of ${\rm G}$.
We will often need to consider the groups $\glr$, when $r=2n$ and $r=2n+1$. 
In this case we denote the standard maximal unipotent subgroup by 
${\rm N}_r$. Let ${\rm M}$ be the space 
of all $n\times n$ matrices and $V$ the subspace of 
all upper triangular $n\times n$ matrices. 
 
Let $ \psi_v$ be a nontrivial additive character of $F_v$. We may view 
$\psi_v$ as a character of $N_v$ by setting
\[
\psi_v(n) = \psi_v\left(\displaystyle{\sum_{i=1}^{n-1}}n_{i,i+1}\right),
\]
for $n\in N_v$. Similarly, a global additive character $\psi$ of $\af$ can be
viewed as a character of $N$.
For a representation $\pi_v$ of $G_v$ on a vector space $U$, let 
$ U_{\psi_v}^{*} $ be the space of all linear forms on $U$ satisfying
\[
\lambda(\pi_v(n)v) = \psi_v(n)\lambda(v),
\]
for all $ v \in U$ and $ n \in N_v$. If $\dim U_{\psi_v}^{*} = 1$, we denote by $\mathcal{W}(\pi_v, \psi_v)$ 
the Whittaker model of $ \pi_v$, which is the space of functions $W_v(g)$ on $G_v$ defined by
\[
W_v(g) = \lambda(\pi_v(g)v), 
\]
where $ v \in U$ and $ \lambda \in U_{\psi}^{*} $.  Note that $G_v$ acts on $ \mathcal{W}(\pi_v, \psi_v)$ 
by right translation and we have $ W_v(ng) = \psi_v(n) W_v(g) $ for  $ n \in N_v$ and  $ g \in G_v. $ 
We say that a representation $ \pi_v$ is generic if it is irreducible and $ \dim U_{\psi_v}^{*} = 1$.

 If $(\pi_v, U)$ is a representation of $G_v$, $(\tilde{\pi}_v, \tilde{U}) $ will denote the contragredient 
 representation of $ (\pi_v, U) $. If $\pi_v$ has a Whittaker model and $W_v \in \mathcal{W}(\pi_v, \psi_v)$, we define the function 
 $\tilde{W}_v$ on $G_v$ by $\tilde{W}_v(g) = W_v(w_{n}~^{\iota}g)$, where $^{\iota}g =~^tg^{-1}$ and
\[w_n = \left( \begin{array}{ccc} &&1 \\ & \iddots & \\ 1&& \end{array} \right).\] The space 
of functions $\mathcal{W}(\tilde{\pi}_v, \bar{\psi}_v) = \{ \tilde{W}\,|\,W \in \mathcal{W}(\pi_v, \psi_v)\}$ 
gives the Whittaker model of $\tilde{\pi}_{v}$.

Let $ \omega_{\pi_v} $ be the central character of $ \pi_v $, if it exists. We say that an irreducible smooth representation $ (\pi_v, U) $ of $G_v$ is square integrable if its central character is unitary and 
\[
\int_{Z_v\backslash G_v} |f(\pi (g)u)|^{2} dg < \infty,
\]
for all $u \in U $ and $ f \in \tilde{U} $. A smooth irreducible representation $\pi$ of $G_v$ is called quasi-square
integrable if it becomes square integrable after twisting by a suitable quasi-character of $G_v$.

\section{The integral representation of Jacquet and Shalika} \label{irm}
We review the theory of the integral representation for the exterior square $L$-function given 
by Jacquet and Shalika in \cite{JaSh90a}, where  several of its important properties were also proved.

\subsection{The Local Theory}\label{irmlocal}
Let $ \pi_v$ be an irreducible generic representation of $\glrfv$. 
In \cite{JaSh90a} Jacquet and Shalika give an integral representation for the exterior square $L$-function 
$L(s, \pi_v,\wedge^{2}), $ using certain families of integrals. If $r$ is even, say $ r = 2n $, we let
\begin{align}\label{J(even)}
\notag J(s, W_v, \phi_v) = 
\int_{N_v \backslash G_v}\int_{V_{v}\backslash M_{v}} & W_v\left(\sigma\left(\begin{array}{cc} I_{n} & X \\ 0 & I_{n} \end{array}\right)
\left(\begin{array}{cc} g & 0 \\ 0 & g \end{array}\right) \right) \hspace{2cm}\\
&\psi_v(-\tr X)dX \phi_v(e_{n}g) |{\det g}|_v^{s} dg,
\end{align}
for each $W_v\in\wpiv$ and $\phi_v$ in ${\mathcal S}(F_v^n)$, where $s$ is in $\C$, and 
$\sigma $ is the permutation given by
\begin{center} 
$ \sigma = \left( \begin{array}{ccccccccc} 
1 & 2 & \cdots & n & | & n+1 & n+2 & \cdots & 2n \\
1 & 3 & \cdots & 2n-1 & | & 2 & 4 & \cdots & 2n 
\end{array} \right) $.
\end{center}
If $r$ is odd, say $ r = 2n+1 $, we consider  
\begin{align}\label{J(odd)}
\notag J(s, W_v) = 
\int_{N_{v} \backslash G_{v}}\int_{V_{v}\backslash M_{v}}&W_v\left(\sigma\left(\begin{array}{ccc} I_{n} & X & 0\\ 0 & I_{n} & 0\\ 0 & 0 & 1 \end{array}\right)
\left(\begin{array}{ccc} g & 0 & 0 \\ 0 & g & 0\\ 0 & 0 & 1 \end{array}\right) \right) \hspace{1cm}\\
 &\psi_v(-\tr X)dX |{\det g}|_v^{s-1} dg, 
\end{align}
for each $W_v$ in $\wpiv$, where  $ \sigma $ is the permutation given by
\begin{center}
$ \sigma = \left( \begin{array}{cccccccccc} 
1 & 2 & \cdots & n & | & n+1 & n+2 & \cdots & 2n & 2n+1 \\
1 & 3 & \cdots & 2n-1 & | & 2 & 4 & \cdots & 2n & 2n+1
\end{array} \right) $.
\end{center}
Combining Proposition 1 of Section 7 and Proposition 3 of Section 9 of \cite{JaSh90a},
we can state the following result.
\begin{proposition} \label{js-int-conv}
Let $\pi_v$ be an irreducible unitary generic representation of $\glrfv$. There exists $\eta > 0 $ 
such that the integrals $ J(s, W_v, \phi_v) $ (resp. $J(s, W_v))$ converge absolutely for 
$\text{Re}(s) >  1 - \eta$.
\end{proposition}
We can use Proposition 2 of Section 7 and Proposition 4 of Section 9 of \cite{JaSh90a} to
obtain the following Proposition for unramified representations.
\begin{proposition}\label{unramvector} Suppose that $F_v$ is a $p$-adic
field and that $\pi_v$ is an unramified representation of $\glrfv$. If $r$ is even
(resp. odd) we can choose $W_{v}^{0} \in \wpiv$ and $\phi_{v}^{0} \in \cs(F_{v}^{n})$ 
(resp. $W_{v}^{0} \in \wpiv$) such that 
\[
J(s,W_v^0,\phi_v^0)=\lgepiv\quad(\text{resp.}\quad J(s,W_v^0)=\lgepiv).
\]
\end{proposition}
In both the $p$-adic and the archimedean cases, it is not hard to see that the integrals above
can be meromorphically continued to the whole of $\C$. 

Let $F_v$ be a $p$-adic field. 
It is easy to see from the proof of the above theorem that the integrals 
$J(s, W_v, \phi_v)$ (resp. $J(s, W_v)$) are rational functions in $ q_v^{-s} $ (see Proposition 
2.3 of \cite{Kewat11}). For $g$ in $G_v$, define elements $g_1$ in $\gltwonfv$ and 
$g_2$ in ${\rm GL}_{2n+1}(F_v)$ as follows:
\[
g_{1} = \left(\begin{array}{cc} g & 0 \\ 0 & g \end{array}\right)\quad\text{and}\quad 
g_{2} = \left(\begin{array}{cc} g_{1} & 0 \\  0& 1 \end{array}\right).
\]
We have
\[
J(s, \pi_v(g_{1})W_v, R(g)\phi_v) = |\mbox{det} g|_v^{-s} J(s, W_v, \phi_v),
\]
where $R$ denotes the right translation action of $G_v$ on $\mathcal{S}(F_v^{n})$, and
\[
J(s, \pi_v(g_{2})W_v) = |\mbox{det} g|_v^{-s} J(s, W_v).
\]
This shows that the $\C$-vector space $\I(\pi_v)$ generated by the integrals 
$J(s, W_v, \phi_v)$ (resp. $J(s, W_v)$) in $\C[q_v^{-s},q_v^{s}]$) is a fractional ideal of 
$\C[q_v^{-s}, q_v^{s}]$. Since $\C[q_v^{-s}, q_v^{s}]$ is a principal ideal domain, 
the fractional ideal $\I(\pi_v)$ is a principal fractional ideal. We now invoke the following theorem 
of Belt (see Theorem 2.2 of \cite{Belt11}).

\begin{theorem}\label{nv-lf} Let $v$ be any place of $F$.  For each 
$s_0\in \C$, there exist $W_v$ in $\wpiv$ and $\phi_v$ in ${\mathcal S}(F_v^n)$ 
(resp. $W_v$ in $\wpiv$) such that $J(s_0,W_v,\phi_v)\ne 0$ (resp. $J(s_0,W_v)\ne 0$).
\end{theorem}

This theorem extends to arbitrary $s_0$ an earlier result of 
Jacquet and Shalika for $s_0=1$ (see \cite{JaSh90a}).
Using the theorem above we see that $1$ lies in $\I(\pi_v)$. As a result 
we can make the following definition.
\begin{definition}
 We define the exterior square $L$-function $ L_{JS}(s, \pi_v,\wedge^{2})$ as the generator of  
 $\I(\pi_v)$ of the form $\ljsepiv$ = $\frac{1}{P(q_v^{-s})}$, where
 $P(q_v^{-s})$ is a polynomial in $\C[q_v^{-s}]$ and $P(0) = 1$.
\end{definition}
\begin{remark}\label{gcdnotknown} As remarked before, although Proposition \ref{unramvector}
allows us to choose local data so that the integral representation gives the $L$-function $\lgepiv$ when 
$\pi_v$ is unramified, it is by no means clear that this choice yields the $L$-function $\ljsepiv$ defined above.
Indeed, it is only after we prove Theorem \ref{loceqthm} that we will be able to establish this fact, and even then,
only in the even case. In the odd case, we are unable to establish this in this paper.
\end{remark}
\begin{remark} We will also use Belt's result in the archimedean case, but not to define the $L$-function.
We emphasise that in this paper, the archimedean $L$-function $\ljsepiv$ is defined via the Local Langlands
Correspondence (as it was, by Jacquet and Shalika). The question of whether the integral representation
yields the $L$-function $\lgepiv$ for a suitable choice of local data, is an open one, as it is in the 
$p$-adic unramified case. 
\end{remark}
If $F_v$ is a $p$-adic field and $\pi_v$ is square integrable Theorem M of \cite{Kewat11} goes further.
\begin{theorem} \label{sq-int-JS}
Let $ \pi_v$ be an irreducible smooth square integrable representation of $\glrfv$, where $F_v$ is a $p$-adic field. 
Then the $L$-function $L_{JS}(s, \pi_v, \wedge^{2}) $ is regular in the region $ \text{Re}(s) > 0 $, if $r$ is even, 
 and in the region $ \text{Re}(s) \geq 0 $, if $r$ is odd.
\end{theorem}
\begin{remark} In Theorem N of \cite{Kewat11}, the first author also proved the non-vanishing of
the local integrals $\jswvphiv$ and $\jswv$ in $\res>0$ for square integrable
representations over a $p$-adic field.
\end{remark}
If $F_v$ is an archimedean local field, we can extract the following proposition from \cite{JaSh90a},
or from the more explicit calculation in \cite{Belt11} (see Proposition 3.4 and the proof of Theorem 2.2
in Section 3.4).
\begin{proposition} \label{arch-fp}
Let $a$ and $b$ be real numbers. There is a finite set of 
points $P(a,b)$ in the strip $ a \leq \res \leq b$ (independent of the choice of $W_v$ and $\phi_v$) 
such that the set of poles of the integrals $\jswvphiv$ (resp. $\jswv$) is contained in $P(a,b)$.
\end{proposition}
\begin{proof}
 It is easy to see from the proof of Proposition 1 of Section 7 and Proposition 3 of Section 9 
 of \cite{JaSh90a} that the integral  $\jswvphiv$ (resp. $\jswv$) is a finite sum of products of entire 
 functions and Tate integrals. The exponents of the quasi-characters occurring in the Tate integrals 
 are finite in number and are independent of the choice of $W_v$ and $\phi_v$ by Proposition 6 of 
 \cite{JaSh90a}. Since these exponents determine the poles of Tate integrals, and since the latter
have at most a finite number of poles in any vertical strip (page 155 of \cite{JaSh90a}), the existence
of the finite set $P(a,b)$ follows.
\end{proof}

\subsection{The Global Theory}\label{irmglobal}
As before, let $F$ be a number field. Let $\Phi$ be a function in $ \cs(\af^{n})$, the space of 
Schwartz-Bruhat functions on $\af^n$. We denote by ${\rm P}_{n-1,n}$ the parabolic 
subgroup of type $(n-1, 1)$ in ${\rm G}$. Let $\pi$ be a unitary cuspidal automorphic representation 
of ${\rm GL}_r$. We denote by $\omega_{\pi}$ the central character of $\pi$. For a non-trivial 
additive character $\psi$ of $\af/F$ and a form $\varphi$ in the space of $\pi$, we consider, when 
$r = 2n$, the integral
\begin{align} \label{global_inteven}
\notag I(s, \varphi, \Phi) = \int_{{\rm G}(F) \backslash G/Z)} \int_{M(F) \backslash M} & \varphi \left(\left(\begin{array}{cc} I_{n} & X \\ 0 & I_{n} \end{array}\right) \left(\begin{array}{cc} g & 0 \\ 0 & g \end{array}\right) \right)\\& 
\psi(\tr X)dX E(g, \Phi, s)dg,
\end{align}
where $  E(g, \Phi, s)$ is the Eisenstein series 
\[
  E(g, \Phi, s) = \displaystyle{\sum_{\gamma \in {\rm P}_{n-1,n}(F) \backslash {\rm G}(F)}}f(\gamma g, s),
\]
with
\[
 f(g, s) = |{\det g}|^{s} \int_{\af^{\times}} \Phi(e_{n}ag) |a|^{ns} \omega_{\pi}(a)da.
\]

If $r = 2n+1$, consider 
\begin{align} \label{global_intodd}
 \notag I(s, \varphi)=
\int_{G(F) \backslash G} \int_{F^{n} \backslash\af^{n}} \int_{M(F) \backslash M} \varphi &\left(\left(\begin{array}{ccc} I_{n} & X & Y \\ 0 & I_{n} & 0\\ 0 & 0 & 1\end{array}\right) \left(\begin{array}{ccc} g & 0 & 0 \\ 0 & g & 0\\ 0 & 0 & 1 \end{array}\right) \right)\\& \psi(\tr X)dX dY |{\det g}|^{s-1}dg.
\end{align}
\indent
Jacquet and Shalika have shown that the integral $  I(s, \varphi) $ converges absolutely for all $s$ (see Proposition 1 of Section 9 of \cite{JaSh90a}). They have also shown that the integral $I(s, \varphi, \Phi)$ converges for all $s$ except at the singularities of the Eisenstein series (see Section 5 of \cite{JaSh90a}). Lemma 4.2 of \cite{JASh811} shows that the
following theorem holds.
\begin{theorem}
 The Eisenstein series $  E(g, \Phi, s) $ is absolutely convergent for $\text{Re}(s) > 1$. It has a meromorphic continuation to the entire complex plane and satisfies the functional equation
\[ 
E(g, \Phi, s) = E(^{\iota}g, \hat{\Phi}, 1-s). 
\]
\end{theorem}

We will need the two Weyl elements
\[
w_{2n} = \left( \begin{array}{ccc} &&1 \\ & \iddots & \\ 1&& \end{array} \right)\quad{and}\quad
w_{n,n} = \left( \begin{array}{cc} 0 &I_{n}\\ I_{n}& 0\end{array} \right)
\]
in $GL_{2n}$. From the above theorem and the proof of Proposition 1 of 
Section 9 of \cite{JaSh90a} (see page 220), we get the following theorem.
\begin{theorem}\label{global_funteq}
The integrals \eqref{global_inteven} and \eqref{global_intodd} satisfy the functional equations
\[ I(s, \varphi, \Phi) = I(1-s, \rho(w_{n,n})\tilde{\varphi}, \hat{\Phi}) \]
and \[ I(s, \varphi) = I(1-s, \varphi'), \]
where $\tilde{\varphi}(g) = \varphi(^\iota g)$, $\varphi'$ is a suitable translate of 
$\tilde{\varphi}$, and $\rho$ denotes the right translation action. 
\end{theorem}

We now define global analogues of the local integrals appearing in Subsection
\ref{irmlocal}. Let
\begin{equation*}
 W_{\varphi}(g) = \int_{N_{r}(F) \backslash N_{r}} \varphi(ug)\psi(u)du 
\end{equation*}
be the Whittaker function associated to $\varphi$, where, as before,
 we view $\psi$ as a character of $N_{r}$ by setting
\begin{equation*}
\psi(u) = \displaystyle{\prod_{j=1}^{r-1}}\psi(u_{j,j+1}).
\end{equation*}
If $r = 2n$, consider
\begin{align*}
 J(s, W_{\varphi}, \Phi) = \int_{N\backslash G}\int_{V'\backslash M} & W_{\varphi}\left(\sigma\left(\begin{array}{cc} I_{n} & X \\ 0 & I_{n} \end{array}\right)
\left(\begin{array}{cc} g & 0 \\ 0 & g \end{array}\right) \right) \hspace{2cm}\\
&\psi(\tr X)dX \Phi(e_{n}g) |{\det g}|^{s} dg,
\end{align*}
where ${\rm V}'$ is the subspace of strictly upper triangular matrices in ${\rm M}$ and $ \sigma $ is the permutation given by
\begin{equation*}
 \sigma = \left( \begin{array}{ccccccccc} 
1 & 2 & \cdots & n & | & n+1 & n+2 & \cdots & 2n \\
1 & 3 & \cdots & 2n-1 & | & 2 & 4 & \cdots & 2n 
\end{array} \right).
\end{equation*}
If $r = 2n+1 $, consider 
\begin{align*} 
 J(s, W_{\varphi}) = 
\int_{N\backslash G}\int_{V'\backslash M} & W_{\varphi}\left(\sigma\left(\begin{array}{ccc} I_{n} & X & 0\\ 0 & I_{n} & 0\\ 0 & 0 & 1 \end{array}\right)
\left(\begin{array}{ccc} g & 0 & 0 \\ 0 & g & 0\\ 0 & 0 & 1 \end{array}\right) \right) \\
 &\psi(\tr X)dX |{\det g}|^{s-1} dg, 
\end{align*}
where  $ \sigma $ is the permutation given by
\begin{center}
$ \sigma = \left( \begin{array}{cccccccccc} 
1 & 2 & \cdots & n & | & n+1 & n+2 & \cdots & 2n & 2n+1 \\
1 & 3 & \cdots & 2n-1 & | & 2 & 4 & \cdots & 2n & 2n+1
\end{array} \right) $.
\end{center}
We can combine Proposition 5 of Section 6 and Proposition 2 of 
Section 9 of \cite{JaSh90a} to get the following proposition.
\begin{proposition} \label{globalint_eqlty} For ${\rm Re}(s)$ sufficiently large we 
obtain the following equalities.
\begin{enumerate}
\item The integral $J(s, W_{\varphi}, \Phi)$ converges absolutely and
\[
I(s, \varphi, \Phi) = J(s, W_{\varphi}, \Phi).
\]
\item The integral $J(s, W_{\varphi})$ converges absolutely and
\[
I(s, \varphi) = J(s, W_{\varphi}).
\]
\end{enumerate}
\end{proposition}
The global integrals are easily related to the local integrals for decomposable vectors. 
If $W_{\varphi} = \displaystyle{\prod_{v}} W_{v}$ and $\Phi = \displaystyle{\prod_{v}}\Phi_{v} $, 
where $v$ runs over all the absolute values of $F$, then for ${\rm Re}(s)$  sufficiently large,
we have
\begin{equation*}
 J(s, W_{\varphi}, \Phi) = \displaystyle{\prod_{v}}J(s, W_{v}, \Phi_{v}) ~\text{and}~ J(s, W_{\varphi}) = \displaystyle{\prod_{v}}J(s, W_{v}).
\end{equation*}

\section{The Langlands-Shahidi method}\label{lsm}
We briefly recall the results from the  Langlands-Shahidi method for the 
exterior square $L$-function.

\subsection{The Local theory}\label{lsmlocal}
Let $ L_{Sh}(s, \pi_v, \wedge^{2}) $ be the exterior square $L$-function defined by 
Langlands-Shahidi method. For a tempered representation $\pi_v$ of $\glrfv$ over 
a $p$-adic field $F_v$, the $L$-function 
$L_{Sh}(s, \pi_v, \wedge^{2}) $ is defined as the inverse of a certain unique polynomial 
$P(q_v^{-s})$ in $q_v^{-s}$ satisfying $P(0) = 1$, and such that $P(q_v^{-s})$ is the 
numerator of a certain gamma factor 
$\gamma(s, \pi_v, r_{i}, \psi_v)$ defined in \cite{Shahidi90}. We refer to \cite{Shahidi90} for the precise definition of 
$L_{Sh}(s, \pi_v, \wedge^{2}) $. 
\begin{proposition} {\rm \cite[Proposition 7.2]{Shahidi90}} \label{temp-Sh}
If $\pi_v$ is a tempered representation of $\glrfv$, where $F_v$ is a $p$-adic field, the local $L$-function $ L_{Sh}(s, \pi_v, \wedge^{2}) $ is holomorphic in the region $\text{Re}(s) > 0 $.
\end{proposition} 
If $F_v$ is an archimedean local field the $L$-functions are known to have the following form.
\begin{proposition} \label{arch-Sh}
The $L$-function $ L_{Sh}(s, \pi_v, \wedge^{2}) $ is a product of Gamma functions of the form 
$c\pi^{-s/2}\Gamma(\frac{s + b}{2}) $ for constants $c\ne 0$ and $b$ in $\C$.
\end{proposition}

\subsection{The Global Theory}\label{lsmglobal}
Let $\pi = \otimes'_{v}\pi_{v}$ be a unitary cuspidal representation of $GL_r$. We define the completed (global) $L$-function as
\begin{equation*}
L_{Sh}(s, \pi, \wedge^{2}) = \displaystyle{\prod_{v}}L_{Sh}(s, \pi_{v}, \wedge^{2}),
\end{equation*} 
where $v$ runs over all the absolute values of $F$. Combining Propositions 3.1 and 3.4 of \cite{Kim99} we obtain:
\begin{proposition}
The $L$-function $L_{Sh}(s, \pi, \wedge^{2})$ is holomorphic for $\text{Re}(s) > 1$.
\end{proposition}
Combining Theorem 3.5 and Proposition 3.6 of \cite{Kim99}) gives 
\begin{theorem}\label{kimthm}
If $n$ is even and $\pi$ is non self dual, or if $n$ is odd, the $L$-function $L_{Sh}(s, \pi, \wedge^{2})$ is entire.
\end{theorem} 

A theorem of Shahidi in \cite{Shahidi90} shows that the $L$-function $L_{Sh}(s, \pi, \wedge^{2})$ satisfies 
a functional equation.
\begin{theorem}{\rm \cite[Theorem 7.7]{Shahidi90}} \label{global_funteqsh}
The $L$-function $L_{Sh}(s, \pi, \wedge^{2})$ admits a meromorphic continuation to the entire complex plane 
and satisfies a functional equation
\begin{equation}\label{lshfnaleqn}
L_{Sh}(s, \pi, \wedge^{2}) = \epsilon_{Sh}(s, \pi, \wedge^{2}) L_{Sh}(1-s, \tilde{\pi}, \wedge^{2}),
\end{equation} 
where the function $\epsilon_{Sh}(s, \pi, \wedge^{2})$ is entire and non-vanishing, and $\tilde{\pi}$ denotes the representation contragredient to $\pi$.
\end{theorem}
When $v$ is a finite place of $F$, Henniart has shown that the Langlands-Shahidi and Galois $L$-functions 
are equal (this was already known for the finite unramified places by \cite{Shahidi81}). Combining this with 
a result of Shahidi for the archimedean places in \cite{Shahidi85}, we can state the following theorem.
\begin{theorem}\label{henniart}
 Let $\pi_v$ be a smooth irreducible representation of $\glrfv$. Then
\[ L_{Sh}(s, \pi_v, \wedge^{2}) = L(s, \pi_v, \wedge^{2}).
\]
\end{theorem}

\section{A Lemma about $\epsilon$-factors} \label{esqfcr}
In this section we prove a lemma (Lemma \ref{monemon}) about 
proportionality factors that appear when relating the 
local exterior square $L$-function of the contragredient representation to 
the integral representation involving Whittaker functions ``dual" to the spherical
vector. As we will see in Section \ref{locfneqn}, these proportionality factors are
the $\epsilon$-factors that appear in the local functional equation,
and Lemma \ref{monemon} asserts that this factor must be 
entire and non-vanishing.
We give the detailed proof only in the even case since the odd case follows 
from the same arguments and is simpler.

Let $\pi=\otimes'_v\pi_v$ be a unitary cuspidal automorphic representation of ${\rm GL}_r$,
with $r=2n$.
Let $S_{\infty}$ denote the set of archimedean places of $F$ and $S_{r} $ 
the set of finite places $v$ for which $\pi_v$ is not unramified. We set 
$S=S_{\infty} \cup S_{r}$. With the notation of Proposition \ref{unramvector} and 
using Theorem \ref{henniart}, we have the following proposition which yields
the equality of the $L$-functions for $v\notin S$.
\begin{proposition}\label{unramequal} Let $\pi_v$ be an unramified representation
of a $p$-adic field $F_v$. Then
\begin{equation}\label{unrami}
J(s,W_v^0,\phi_v)~(\text{resp.} J(s,W_v^0))=\lshepiv=\lgepiv.
\end{equation}
\end{proposition}
Let $ v_0 \in S_{r}$. Since the local $L$-function is defined as a generator of 
$\I(\pi_{v_0})$, there exist Whittaker functions $W_{i,v_0}$ and Schwartz-Bruhat functions 
$\phi_{i,v_0}$ such that
\begin{equation}\label{rami1}
  L_{JS}(s, \pi_{v_0}, \wedge^{2}) = \sum_{i=1}^{n_{v_0}}J(s, W_{i,v_0}, \phi_{i,v_0}).
\end{equation}
Applying the same reasoning to the contragredient representation $\tilde{\pi}_{v_0}$, we get
\begin{equation}\label{contra1}
 \sum_{i=1}^{n_{v_0}}J(s, \rho(w_{n,n})\tilde{W}_{i,v_0},\hat{\phi}_{i,v_0}) = M_{1}(q_{v_0}^{-s}) L_{JS}(s, \tilde{\pi}_{v_0}, \wedge^{2}),
\end{equation}
where $M_{1}(X)$ is a polynomial in $\C[X,X^{-1}]$. Similarly, there exist $W'_{i,v_0} $ and $\phi'_{i,v_0}$ such that
\begin{equation}\label{rami2}
  L_{JS}(s, \tilde{\pi}_{v_0}, \wedge^{2}) = \sum_{i=1}^{m_{v_0}}J(s, \rho(w_{n,n})\tilde{W}'_{i,v_0},\hat{\phi}'_{i,v_0})
\end{equation}
and
\begin{equation}\label{contra2}
 \sum_{i=1}^{m_{v_0}}J(s, W'_{i,v_0}, \phi'_{i,v_0}) = M_{2}(q_{v_0}^{-s}) L_{JS}(s, \pi_{v_0}, \wedge^{2}),
\end{equation}
where $M_{2}(X)$ is a polynomial in $\C[X,X^{-1}]$. In what follows, by a monomial in 
$\C[X,X^{-1}]$ we will mean a polynomial of the form $cX^m$, with $m\in \Z$.

\begin{lemma}\label{monemon}
The polynomials $M_{1}$ and $M_{2}$ are monomials in $q_{v_0}^{-s}$.
\end{lemma}
\begin{proof}
For $v \in S_r\smallsetminus \{v_0\},$ we make a specific choice of $W_{v}$ and $\phi_{v}$ such 
that the integrals $J(s, W_v, \phi_v)$ are not identically zero. For $v\in S_{\infty}$ we may
take $W_v$ and $\phi_v$ arbitrary, again such that $J(s,W_v,\phi_v)$ is not identically zero. Let
\begin{equation*}
 W_{i} =  W_{i, v_0} \cdot \prod_{v \in S \smallsetminus \{v_0\}} W_{v}\cdot \prod_{v \notin S} W_{v}^{0}
\end{equation*}
and
\begin{equation*}
 \Phi_{i} =  \phi_{i, v_0} \cdot \prod_{v \in S \smallsetminus \{v_0\}} \phi_{v}\cdot \prod_{v \notin S} \phi_{v}^{0}.
\end{equation*}
Let
\begin{align} \label{F_1}
\notag F_{1}(s, \pi) & = \sum_{i = 1}^{n_{v_0}} J(s, W_{i}, \Phi_{i})\\
& = L_{JS}(s, \pi_{v_0}, \wedge^{2}) \cdot \prod_{v \in S_r \smallsetminus \{v_0\}} J(s, W_{v}, \phi_{v}) \cdot \prod_{v \notin S} 
J(s, W_{v}^{0}, \phi_{v}^{0}) \hspace{.5cm} (\text{using}~ \eqref{rami1})
\end{align}
and
\begin{align} \label{F_2}
\notag F_{2}(s, \tilde{\pi}) & = \sum_{i = 1}^{m_{v_0}} J(s, \rho(w_{n,n})\tilde{W}_{i}, \hat{\Phi}_{i}) \\
& = M_{1}(q_{v_0}^{-s})L_{JS}(s, \tilde{\pi}_{v_0}, \wedge^{2}) \cdot \prod_{v \in S_r\smallsetminus \{v_0\}} J(s, \rho(w_{n,v})\tilde{W}_{v}, \hat{\phi}_{v}) \cdot & \prod_{v \notin S} J(s, \tilde{W}_{v}^{0}, {\hat{\phi}}_{v}^{0})\\\notag && (\text{using}~ \eqref{contra1})
\end{align}
From Theorem \ref{global_funteq}, Proposition \ref{globalint_eqlty} and Theorem \ref{global_funteqsh}, we have
\begin{equation} \label{gfeqr}
 \dfrac{F_{1}(s, \pi)}{L_{Sh}(s, \pi, \wedge^{2})} = \dfrac{F_{2}(1 - s, \tilde{\pi})}{\epsilon_{Sh}(s, \pi, \wedge^{2}) L_{Sh}(1-s, \tilde{\pi}, \wedge^{2})}.
\end{equation}
This gives, using equations \eqref{unrami}, \eqref{F_1} and \eqref{F_2},
\begin{align} \label{eq_fin1}
 \notag \dfrac{L_{JS}(s, \pi_{v_0}, \wedge^{2})}{L_{Sh}(s, \pi_{v_0}, \wedge^{2})}\cdot \prod_{v \in S \smallsetminus \{v_0\}} \dfrac {J(s, W_v, \pi_v)}{L_{Sh}(s, \pi_{v}, \wedge^{2})} = &~ \dfrac{M_{1}(q_{v_0}^{-s})}{\epsilon_{Sh}(s, \pi, \wedge^{2})} \cdot \dfrac{L_{JS}(1-s, \tilde{\pi}_{v_0}, \wedge^{2})}{L_{Sh}(1-s, \tilde{\pi}_{v_0}, \wedge^{2})} \\& \times\prod_{v \in S \smallsetminus \{v_0\}}  \dfrac {J(1-s, \rho(w_{n,n})\tilde{W}_v, \hat{\pi}_v)}{L_{Sh}(1-s, \tilde{\pi}_{v}, \wedge^{2})}.
\end{align}
By applying the same reasoning as above to $\tilde{\pi}$, we get
\begin{align}\label{eq_fin2}
\notag &\dfrac{L_{JS}(s, \tilde{\pi}_{v_0}, \wedge^{2})}{L_{Sh}(s, \tilde{\pi}_{v_0}, \wedge^{2})} \cdot \prod_{v \in S \smallsetminus \{v_0\}}  \dfrac {J(s, \rho(w_{n,n})\tilde{W}_v, \tilde{\pi}_v)}{L_{Sh}(s, \tilde{\pi}_{v}, \wedge^{2})} = \\ 
&\omega_{\pi}(-1)\dfrac{M_{2}(q_{v_0}^{-s})}{\epsilon_{Sh}(s, \tilde{\pi}, \wedge^{2})}  \cdot \dfrac{L_{JS}(1-s, \pi_{v_0}, \wedge^{2})}{L_{Sh}(1-s, \pi_{v_0}, \wedge^{2})} \times\prod_{v \in S \smallsetminus \{v_0\}} \dfrac {J(1-s, W_v, \pi_v)}{L_{Sh}(1-s, \pi_{v}, \wedge^{2})},
\end{align}
where the factor $\omega_{\pi}(-1)$ arises because $\hat{\hat{\phi}}(x)=\phi(-x)$.
Combining \eqref{eq_fin1} and \eqref{eq_fin2}, we obtain
\begin{equation*}
 M_{1}(q_{v_0}^{-s})M_{2}(q_{v_0}^{-(1-s)}) = \omega_{\pi}(-1) \epsilon_{Sh}(s, \pi, \wedge^{2}) \epsilon_{Sh}(1-s, \tilde{\pi}, \wedge^{2})
 =\pm 1.
\end{equation*}
Hence, the polynomials $M_{1}$ and $M_{2}$ are non-vanishing, and it follows that they must be 
monomials in $q_{v_0}^{-s}$. 
\end{proof} 
The lemma also holds for the case $r=2n+1$. Indeed, if we have
 \[
L_{JS}(s, \pi_{v_0}, \wedge^{2}) = \sum_{i=1}^{n_{v_0}}J(s, W_{i,v_0})
\]
and
\[
 \sum_{i=1}^{n_{v_0}}J(s, \tilde{W}_{i,v_0}) = M(q_{v_0}^{-s}) L_{JS}(s, \tilde{\pi}_{v_0}, \wedge^{2}),
\] 
where $M(X)$ is a polynomial in $\C[X,X^{-1}]$, the same reasoning as in the even case shows that
that $M(q_{v_0}^{-s})$ will actually be a monomial in $q_{v_0}^{-s}$.

\section{The proof of the main theorem}\label{proof}
We are now ready to establish Theorem \ref{sqinteqthm}.
We will concentrate on the case $r=2n$ and omit the case $r=2n+1$, since the proofs 
follow along similar lines and are, in fact, somewhat easier.

We start with a proposition that allows us to embed a square integral representation 
as the local component of a (global) cuspidal automorphic representation. We use 
a weaker form of Lemma 6.5 of Chapter 1 of \cite{ArCl89}.
\begin{proposition} \label{arthur-clozel-embed}
Let $v_0$ be a place of $F$. If $\pi_{v_0}$ is a square integrable 
representation of ${\rm GL}_r(F_{v_0})$, there exists a cuspidal automorphic 
representation $\Pi=\otimes'_v\Pi_v$ of $GL_r$ such that $\Pi_{v_0} \simeq \pi_{v_0}$.
\end{proposition}
We will also need Lemma 5 of \cite{Kable04}.
\begin{lemma} Let $K$ be a $p$-adic field. There exists a number field $F$ and a place 
$v_0$ of $F$ such that $F_{v_0}=K$, 
where $v_0$ is the unique place of $F$ lying over the rational prime $p$.
\end{lemma}
We can now embark on the proof of Theorem \ref{sqinteqthm}.
\begin{proof}
Let $r=2n$. We can and will assume that $\Pi$ is unitary in the above proposition. 
Let $\tau$ be a square integrable representation of $\gln(K)$. We 
choose $F$ as in the lemma above so that $F_{v_0}=K$. Hence, we 
may view $\tau$ as a (square integrable) representation $\pi_{v_0}$
of ${\rm GL}_{2n}(F_{v_0})$. Using the proposition above, we can find a 
cuspidal automorphic $\Pi$ of ${\rm GL}_{2n}$ with $\Pi_{v_0}=\pi_{v_0}$. 
We now consider the quotients
\[
 G_{1}(s, \Pi) = \dfrac{F_{1}(s, \Pi)}{L_{Sh}(s, \Pi, \wedge^{2})}~\text{and}~
 G_{2}(s, \tilde{\Pi}) = \dfrac{F_{2}(s, \tilde{\Pi})}{M_1(q_{v_{0}}^{-(1-s)})L_{Sh}(s, \tilde{\Pi}, \wedge^{2})},
\]
where $F_1$, $F_2$ and $M_1$ are as in equations \eqref{F_1} and \eqref{F_2}. From equation \eqref{gfeqr}, we have
\begin{equation} \label{gfeqG}
 G_1(s, \Pi) = \eta(s, \Pi)G_2(1-s, \tilde{\Pi}),
\end{equation}
where $\eta(s, \Pi) = M_1(q_{v_{0}}^{-s})/\epsilon_{Sh}(s, \Pi, \wedge^{2})$ is an 
entire function without zeros. On the other hand, at the places where $\Pi$ is unramified
the local integrals in the numerator and the $L$-functions in the denominator cancel each other out. Hence,
\begin{equation*}
 G_{1}(s, \Pi) = \dfrac{L_{JS}(s, \pi_{v_0}, \wedge^{2})}{L_{Sh}(s, \pi_{v_0}, \wedge^{2})}\cdot \prod_{v \in S_0 \smallsetminus \{v_{0}\}} \dfrac {J(s, W_v, \Phi_v)}{L_{Sh}(s, \Pi_{v}, \wedge^{2})}\cdot\prod_{v \in S_\infty} \dfrac {J(s, W_v, \Phi_v)}{L_{Sh}(s, \Pi_{v}, \wedge^{2})}
\end{equation*}
We write this as 
\[
 G_{1}(s, \Pi) = P(s, \Pi) Q_{1}(s, \Pi) R_{1}(s, \Pi),
\]
where
\[
P(s, \Pi) = \dfrac{L_{JS}(s, \pi_{v_0}, \wedge^{2})}{L_{Sh}(s, \pi_{v_0}, \wedge^{2})}~\text{and}~
R_{1}(s, \Pi) = \prod_{v \in S_\infty} \dfrac {J(s, W_v, \Phi_v)}{L_{Sh}(s, \Pi_{v}, \wedge^{2})}
\]
and 
\begin{equation*}
 Q_{1}(s, \Pi) = \prod_{v \in S_0 \smallsetminus \{v_{0}\}} \frac {J(s, W_v, \Phi_v)}{L_{Sh}(s, \Pi_{v}, \wedge^{2})} = \prod_{v \in S_0 \smallsetminus \{v_{0}\}} 
 \frac{\prod_{i=1}^{k_v}(1 - \alpha_i(v)q_v^{-s})}{\prod_{j=1}^{l_{v}}(1 - \beta_j(v)q_v^{-s})},
\end{equation*}
for integers $k_v$ and $l_v$, and complex numbers $\alpha_i(v)$ and $\beta_j(v)$. Note that 
by our assumption on $F$, $(p, q_v) = 1$. Similarly, we have 
\begin{equation*}
 G_{2}(s, \Pi) = P(s, \tilde{\Pi}) Q_{2}(s, \tilde{\Pi}) R_{2}(s, \tilde{\Pi}),
\end{equation*}
where 
\[
 Q_{2}(s, \tilde{\Pi}) = \prod_{v \in S_0 \smallsetminus \{v_{0}\}} \dfrac {J(s, \rho(w_{n,n})\tilde{W}_v, \hat{\Phi}_v)}{L_{Sh}(s, \tilde{\Pi}_{v}, \wedge^{2})} = 
 \prod_{v \in S_0 \smallsetminus \{v_{0}\}} 
 \frac{\prod_{i=1}^{k'_v}(1 - \alpha'_i(v)q_{v}^{-s})}{\prod_{j=1}^{l'_{v}}(1 - \beta'_j(v)q_{v}^{-s})},
\]
for integers $k'_v$ and $l'_v$ and complex numbers $\alpha'_i(v)$ and 
$\beta'_j(v)$, and
\[
R_{2}(s, \tilde{\Pi}) = \prod_{v \in S_\infty} \dfrac {J(s, \rho(w_{n,n})\tilde{W}_v, \hat{\Phi}_v)}{L_{Sh}(s, \tilde{\Pi}_{v}, \wedge^{2})}.
\]
By Theorem \ref{sq-int-JS} and Proposition \ref{temp-Sh}, the functions $P(s, \Pi)$ and $P(s, \tilde{\Pi})$ are regular and non-vanishing in the region $\text{Re}(s) > 0 $. By Proposition \ref{arch-fp} and Proposition \ref{arch-Sh}, the function $R_{1}(s, \Pi)$ and $R_{2}(s, \tilde{\Pi})$ have only finitely many poles in any vertical strip $ a \leq \text{Re}(s) \leq b$. 

To prove Theorem \ref{sqinteqthm} it is enough to prove that the function $P(s, \Pi)$ is entire and nowhere 
vanishing. It will then follow that $P(s,\Pi)$ must be a monomial in $q_{v_0}^{-s}$. 
Since both the $L$-functions $\ljsepiv$ and $\lshepiv$ are normalised
to have numerator $1$, it is immediate that $P(s,\Pi)$ must be identically $1$.

Suppose that $P(s, \Pi)$ has a zero at $s_0$. This means that the function 
$P(s, \Pi)$ also has zeros at $s_0 + 2\pi ik/\log q_{v_0}$,  for all $k \in \Z$. 
We claim that all but finitely many of these zeros must also be zeros of
$G_1(s,\Pi)$. This fails to happen only if all but finitely many zeros are cancelled by the 
poles of $Q_{1}(s, \Pi) R_{1}(s, \Pi)$. Since $R_{1}(s, \Pi)$ can contribute only finitely 
many poles on any line with real part constant, $Q_{1}(s, \Pi)$ must have infinitely 
many poles of this form. On the other hand, the poles of $Q_1(s,\Pi)$ are of the form
$s_j+2\pi il/\log q_v$, for all $l\in \Z$, with $v\in S_0\setminus \{v_0\}$. It follows that 
there is at least one $v$ such that there exist two integers $l_1\ne l_2$ such that 
\[
s_0+2\pi il_1/\log q_v=s_1 + 2\pi ik_1/\log q_{v_0}~\text{and}~
s_0+2\pi il_2/\log q_v=s_1 + 2\pi ik_2/\log q_{v_0}
\]
for some $k_1$ and $k_2$ in $\Z$ (in fact, there are infinitely many distinct integers with this 
property). It follows that $\log q_v/\log q_{v_0}$ is rational, which is absurd since
$(q_v,q_{v_0})=1$ by choice. Thus, for all but finitely many $k$, the points
$s_0 + 2\pi ik/\log q_{v_0}$ are zeros of $G_1(s,\Pi)$.

Since $P(s, \Pi)$ is non-vanishing in the region $\text{Re}(s) > 0$, 
we must have $ \text{Re}(s_0) \leq 0$. From \eqref{gfeqG}, we see that all but finitely many of the
points $1-s_0 + 2\pi ik/ \log q_{v_0}$ are zeros of the function
$G_{2}(s, \tilde{\Pi})$. Since $P(s, \tilde{\Pi})$ is non-vanishing in the region $\text{Re}(s) > 0$, 
these zeros have to be the zeros of $ Q_{2}(s, \tilde{\Pi}) R_{2}(s, \tilde{\Pi})$. Arguing as 
above, these cannot be zeros of $Q_{2}(s, \tilde{\Pi})$ 
for infinitely many $k$. By Proposition \ref{arch-Sh}, the poles of 
$\prod_{v \in S_{\infty}} L_{Sh}(s,\tilde{\Pi}_v, \wedge^2)$ lie along
horizontal lines. Hence, this product can contribute only 
finitely many poles on any line with real part constant. Thus, except for finitely many 
$k$, these zeros must be zeros of 
$\prod_{v \in S_{\infty}} J(s, \rho(w_{n,n})\tilde{W}_v, \hat{\Phi}_v)$,
for every $\rho(w_{n,n})\tilde{W}_v$ and $\hat{\Phi}_v$ (with $v\vert\infty$) such that 
$J(s, \rho(w_{n,n})\tilde{W}_v, \hat{\Phi}_v)$ is not identically zero.
If $\beta = 1-s_{0} + 2 \pi il/\log q_{v_0} $ is one of these zeros of $G_{2}(s, \tilde{\Pi})$, this
contradicts Theorem \ref{nv-lf} which asserts that there are  $W_v$ and $\Phi_v$ such that
$\prod_{v \in S_{\infty}} J(\beta, \rho(w_{n,n})\tilde{W}_v, \hat{\Phi}_v)\ne 0$. 
Hence $P(s, \Pi)$ is non-vanishing.

We now show that $P(s,\Pi)$ must be entire. We rewrite functional equation 
\eqref{gfeqG} as
\begin{equation}\label{feqPQR}
P(s, \Pi) Q_1(s, \Pi) R_1(s, \Pi) = \eta(s, \Pi) P(1-s, \tilde{\Pi}) Q_2(1-s, \tilde{\Pi}) R_2(1-s, \tilde{\Pi}).
\end{equation}
Proposition \ref{arch-Sh} and the form of the local Langlands-Shahidi
$L$-factor at the finite places show us that 
$\prod_{v \in S \smallsetminus v_0}L_{Sh}(s, \Pi_{v}, \wedge^{2})$ 
is nowhere vanishing. By Proposition \ref{js-int-conv}, the function 
$ Q_1(s, \Pi) R_1(s, \Pi)$ is holomorphic in $\text{Re}(s) > 1 - \eta $, for some 
$\eta > 0$, and the function $ Q_2(1-s, \tilde{\Pi})$ $ R_2(1-s, \tilde{\Pi})$ is 
holomorphic in $\text{Re}(s) < \eta$. Hence, the function $G_1(s, \Pi) $ 
is holomorphic in $\text{Re}(s) > 1-\eta$ and in $\text{Re}(s) < \eta$. Suppose that 
$P(s, \Pi)$ has a pole at  $s_0$. This means that the function $P(s, \Pi)$ 
also has poles at $s_{0} + 2\pi ik/\log  q_{v_0}$, $ k \in \Z$. The function 
$P(s, \Pi)$ is holomorphic in the region $\text{Re}(s) > 0$, hence, we obtain 
$ \text{Re}(s_0) \leq 0$. Since $G_{1}(s, \Pi)$ is holomorphic in $ \text{Re}(s) < \eta$, 
$\eta > 0$, these poles must be cancelled by the zeros of $Q_{1}(s, \Pi) R_{1}(s, \Pi)$. 
Arguing as in the the non-vanishing case, these cannot be zeros of $Q_{1}(s, \Pi)$ 
for infinitely many $k$. By Proposition \ref{arch-Sh}, the poles 
of $\prod_{v \in S_{\infty}} L_{Sh}(s,\Pi_v, \wedge^2)$ lie along
horizontal lines. Hence, this product can contribute only 
finitely many poles on any line with real part constant. Thus, except for finitely many 
$k$, these poles must be zeros of $\prod_{v \in S_{\infty}} J(s, W_v, \Phi_v)$,
for every $W_v$ and $\Phi_v$ (with $v\vert\infty$) such that $J(s, W_v, \Phi_v)$ is not 
identically zero. As in the preceeding paragraph, this contradicts Theorem \ref{nv-lf}. 
Hence, $P(s, \Pi)$ must be entire and this completes the proof of Theorem \ref{sqinteqthm}.
\end{proof}

\section{Extensions and applications of the main theorem}\label{extensions}
We now use Theorem \ref{sqinteqthm} to prove Theorems \ref{loceqthm} and \ref{globeqthm}
and give a number of other applications. 

Recall that the symmetric square $L$-function of a representation $\pi_v$ can be defined via the 
local Langlands correspondence as a Galois $L$-function. As before, if $\rho_{F_v}(\pi_v)$ 
corresponds to $\pi_v$ we define
\[
L(s, \pi_v, \text{Sym}^2)=L(s,\text{Sym}^2(\rho_{F_v}(\pi_v))).
\]
As a first application of Theorem \ref{sqinteqthm}, we are able to obtain the following characterisation 
of self dual square integrable representations.
\begin{corollary}\label{symsq}
 Let $\pi_v$ be an irreducible smooth square integrable representation of ${\rm GL}_{2n}(F_{v})$
which has no Shalika functional. Then the symmetric 
square $L$-function $L(s, \pi_v, \text{Sym}^2)$ has a pole at $s= 0$ if and only if $\pi_v$ 
is self dual, that is, if and only if $\pi_v \simeq \tilde{\pi}_v$.
\end{corollary}
\begin{proof}
 We have
\begin{equation} \label{ext-sym}
 L(s, \pi_v \times \pi_v) = L(s, \pi_v, \wedge^2) L(s, \pi_v, \text{Sym}^{2}).
\end{equation}
We know that the $L$-function $L(s, \pi_v \times \pi_v)$ has a pole at $s = 0$ if and only if $\pi_v \simeq \tilde{\pi}_v.$ 
Since $2n$ is even, Corollary 4.4 of \cite{Kewat11} shows that the $L$-function $L_{JS}(s, \pi_v, \wedge^2)$ does not have a pole at $s = 0$. Thus, from Theorems \ref{sqinteqthm} and \ref{jsgaleq}, the $L$-function $L(s, \pi_v, \wedge^2)$ does not have a pole $s = 0$. Hence, the corollary follows from equation \eqref{ext-sym}. The converse is trivial.
\end{proof}

To establish the equality of the exterior square $L$-functions for generic 
representations in the even case we proceed as follows.
We start with proving the equality for quasi-square integrable representations. 
Let $\pi_v$ be a quasi-square integrable representation of $\glrfv $. 
Then $ \pi_v = \pi_0 \otimes \chi$, where $\pi_0$ is a square integrable representation 
and $\chi = \chi_0 |~|^{s_{0}}, \chi_0 $ is an unitary character. Clearly, 
$\pi_0 \otimes \chi_0$ is a square integrable representation. Since
\begin{equation*}
 J(s, W \otimes |~|^{s_{0}}, \phi) = J(s+2s_{0}, W, \phi),
\end{equation*}
we see that
\begin{equation*}
 L_{JS}(s, \pi_v, \wedge^{2}) = L_{JS}(s+2s_0, \pi_0 \otimes \chi_0, \wedge^{2}).
\end{equation*}
Using Theorem \ref{sqinteqthm}, we get
\begin{equation} \label{qsq1}
 L_{JS}(s, \pi_0\otimes\chi_0, \wedge^{2}) = L_{Sh}(s, \pi_0 \otimes \chi_0, \wedge^{2}).
\end{equation}
Let $ \mathcal{A}_{r}(F_v) $ denote the set of isomorphism classes of irreducible admissible representations of 
$\glrfv$, and let $\mathcal{G}_{r}(F_v)$ denote the set of isomorphism classes of $\Phi$-semisimple $r$-dimensional complex representations of Weil-Deligne group $W'_{F_v}$. The local Langlands correspondence (proved by Harris and Taylor in \cite{HaTa01}, see also Henniart \cite{Henniart00}) asserts that for each $r\geq 1$, there exists a bijection 
\begin{equation*}
 \rho_{F_v} : \mathcal{A}_{r}(F_v) \longrightarrow \glrfv,
\end{equation*}
satisfying certain functorial properties. If $\sigma_v$ is a smooth irreducible representation of $\glrfv$ then (
see Theorem \ref{henniart}) has shown that
\begin{equation} \label{henniart-eq}
 L_{Sh}(s, \sigma_v, \wedge^2) = L(s, \sigma_v, \wedge^{2}) = L(s, \wedge^{2}\rho_{F}(\sigma_v)).
\end{equation}
It is easy to see that
\begin{equation}
 \wedge^{2}(\rho_{F}(\sigma_v) \otimes \chi) \simeq \wedge^{2}(\rho_{F}(\sigma_v)) \otimes \chi^{2}.
\end{equation}
Therefore, we have
\begin{align} \label{qsq2}
\notag L(s, \pi_v, \wedge^2) & = L(s, \wedge^{2}\rho_{F}(\pi_0\otimes\chi_0 |~|^{s_0})) =  L(s, \wedge^{2}(\rho_{F}(\pi_0\otimes\chi_0) \otimes |~|^{s_0}))\\
& =  L(s, \wedge^{2}(\rho_{F}(\pi_0\otimes\chi_0)) \otimes |~|^{2s_0}) =  L(s+2s_0, \wedge^{2}\rho_{F}(\pi_0\otimes\chi_0)).
\end{align}
From \eqref{qsq1}, \eqref{henniart-eq} and \eqref{qsq2}, we have
\begin{equation} \label{qsqeq}
 L_{JS}(s, \pi_v, \wedge^{2}) = L_{Sh}(s, \pi_v, \wedge^{2}) = L(s, \pi_v, \wedge^{2}).
\end{equation}
This proves the equality for quasi-square integrable representations which we record
below as a theorem.
\begin{theorem}\label{quasisqintthm}
Let $\pi_v$ be a quasi-square integrable representation of $\glrfv$. Then 
\[ L_{JS}(s, \pi_v, \wedge^{2}) = L_{Sh}(s, \pi_v, \wedge^{2}) = L(s, \pi_v, \wedge^{2}).\]
\end{theorem}

Let $\pi_v$ be an irreducible generic representation of $\glrfv$, where $r=2n$. We now prove Theorem \ref{loceqthm}.
\begin{proof}
It is a theorem of Bernstein and Zelevinsky \cite{BeZe77} that $\pi_v$ is parabolically induced from 
quasi-square integrable representations. Thus, we can write
\begin{center}
$ \pi_v  = \text{Ind}( \pi_{1,v} \otimes \pi_{2,v} \otimes \cdots \otimes \pi_{r,v}), $
\end{center}
where the $ \pi_{i,v} $ are quasi-square integrable representations of ${\rm GL}_{n_{i}}(F)$ with $\sum n_{i} = 2n$. 
In \cite{CoPS942}, Cogdell and Piatetski-Shapiro have proved that 
\begin{equation}\label{gentoqsqJS}
 L_{JS}(s, \pi_v, \wedge^{2}) = \displaystyle{\prod_{i = 1}^{r}} L_{JS}(s, \pi_{i,v}, \wedge^{2}) \prod_{\substack{i <  j \\ j = 2}}^{r} L(s, \pi_{i,v} \times \pi_{j,v}),
\end{equation}
where $ L(s, \pi_{i,v} \times \pi_{j,v}) $ is the Rankin-Selberg $ L $-function of $ \pi_{i,v} \times \pi_{j,v}$.
By the local Langlands correspondence, $\pi_v$ corresponds to 
\[
\rho_{F_v}(\pi_{1,v}) \oplus \rho_{F_v}(\pi_{2,v}) \oplus \cdots \oplus \rho_{F_v}(\pi_{r,v}).
\]
Thus, we have
\begin{equation*}
 L(s, \pi_v, \wedge^{2}) = L(s, \wedge^{2}(\rho_F(\pi_{1,v}) \oplus \rho_F(\pi_{2,v}) \oplus \cdots \oplus \rho_F(\pi_{r,v}))).
\end{equation*}
If $V_1, V_2, \cdots, V_r $ are the spaces on which $\pi_{1,v},\pi_2,\ldots \pi_{r,v}$ act, we can prove that  
\begin{equation*}
 \wedge^{2}(\oplus_{i=1}^{r}V_i)\simeq (\oplus_{i=1}^{r}\wedge^{2}V_i) \bigoplus \left(\oplus_{\substack{i <  j \\ j = 2}}^{r}(V_{i} \otimes V_j)\right).
\end{equation*}
It follows that
\begin{equation}
L(s, \pi_v, \wedge^{2}) = \displaystyle{\prod_{i = 1}^{r}} L(s, \wedge^{2}\rho_F(\pi_{i,v})) \prod_{\substack{i <  j \\ j = 2}}^{r} L(s, \rho_F(\pi_{i,v}) \otimes \rho_F(\pi_{j,v})).
\end{equation}
By the local Langlands correspondence, we have
\begin{equation}
 L(s, \pi_{i,v} \times \pi_{j,v}) = L(s, \rho_{F_v}(\pi_{i,v}) \otimes \rho_{F_v}(\pi_{j,v})).
\end{equation}
Hence, we have
\begin{equation}\label{gentoqsq}
 L(s, \pi_v, \wedge^{2}) = \displaystyle{\prod_{i = 1}^{r}} L(s, \pi_{i_v}, \wedge^{2}) \prod_{\substack{i <  j \\ j = 2}}^{r} L(s, \pi_{i,v} \times \pi_{j,v}),
\end{equation}
From \eqref{qsqeq}, \eqref{gentoqsqJS} and \eqref{gentoqsq}, we have
\begin{equation}
  L_{JS}(s, \pi_v, \wedge^{2}) = L(s, \pi_v, \wedge^{2}).
\end{equation}
This completes the proof of Theorem \ref{loceqthm}.
\end{proof}
Theorem \ref{globeqthm} is an immediate consequence of the Theorems 
\ref{sqinteqthm} and \ref{loceqthm}, combined with the results of Shahidi in 
\cite{Shahidi85} for the archimedean places. It allows us to obtain the analytic properties 
of $\ljsepi$, when they are known for $\lshepi$, and conversely. The analytic properties of 
$\lshepi$ due to Kim and Shahidi were recorded in 
Theorems \ref{kimthm} and \ref{global_funteqsh} from Section \ref{lsm}. We also use 
the theorem of Gelbart and Shahidi in \cite{GeSh01} which shows that $\lshepi$ is bounded 
in vertical strips.
\begin{corollary}\label{analyticjseven} 
If $\pi$ is a unitary cuspidal automorphic representation of $GL_{2n}$ which is 
not self dual, the $L$-function $\ljsepi$ is entire, satisfies the functional equation 
\eqref{lshfnaleqn}, and is bounded in vertical strips.
\end{corollary}
In the odd case, we can make only a weaker statement about the integrals since the equality of
the local $L$-functions has not been established even for unramified representations. However,
this statement should suffice for many, if not most, applications. As in Section \ref{esqfcr}, we
let $S_{\infty}$ denote the archimedean places of $F$ and $S_{r}$ denote the set of places
where $\pi_v$ is not unramified. We denote by $S_{ur}$ the set of finite places where $\pi_v$ is
unramified.
\begin{corollary}\label{analyticjsodd} Let $\pi$ be a unitary cuspidal automorphic representation of 
${GL}_{2n+1}$ such that the local components $\pi_v$ at the finite places are either unramified
or square integrable. Then
\[
\prod_{v\in S_{\infty}\cup S_{ur}}\lshepiv\prod_{v\in S_{r}}\ljsepiv
\]
is entire, satisfies the functional equation \eqref{lshfnaleqn}, and is bounded in vertical strips.
\end{corollary}
To prove the corollaries, we simply choose $W_v$ and $\phi_v$ (resp. $W_v$) 
so that $J(s,W_,\phi_v)$ (resp. $J(s,W_v)$ gives $\lshepiv$ at each finite place.

The facts that $\ljsepi$ has a functional equation and that it is bounded in 
vertical strips are new results. The corollaries above also strengthen the results of 
Theorem 5.2 of \cite{Belt11} where all the ramified and archimedean places are excluded 
for the stated holomorphy result.

For the case when $r$ is even and $\pi$ is self dual, we can use Belt's
theorem for the partial Jacquet-Shalika $L$-function to deduce holomorphy results for 
$\lshepi$. In conjunction with Section 8 of \cite{JaSh90a}, we can obtain the following
corollary.
\begin{corollary}\label{analyticlsh} Assume that $\pi$ is a 
unitary cuspidal automorphic representation of $GL_{2n}$ which is self dual and that 
the local components $\pi_v$ at the archimedean places are tempered.
If the central character $\omega_{\pi}$ is not trivial, then $\lshepi$ is entire.
If $\omega_{\pi}$ is trivial,  $\lshepi$ is holomorphic at all points except possibly
for simple poles at $s=0$ or $s=1$. There will be simple poles if and only if 
$\pi$ has a non-zero global Shalika period.
\end{corollary}
One non-trivial case covered by the corollary above is the following. Let $\sigma$ be
a cuspidal automorphic representation of $\gltwoaq$ associated to an arbitrary holomorphic cusp
form or to a Maass cusp form on the full modular group, and let $\pi$ be the symmetric cube lift of 
$\sigma$ to $\rm{GL}_4(\aq)$. Then the archimedean places of $\pi$ are known to be tempered.

\section{The local functional equation}\label{locfneqn}
The main purpose of this section is to obtain a local functional equation for the 
$L$-functions $\ljsepiv$ using
the same global methods as before. The key point is that we are able to define 
a local $\epsilon$-factor $ \epsilon_{JS}(s, \pi_{v_k}, \psi_{v_k}, \wedge^{2})$. 
Of course, this is conjecturally the same as the $\epsilon$-factor arising in the
Langlands-Shahidi method. We will continue to use the convention that a 
monomial $M(X)$ is an element of $\C[X,X^{-1}]$ of the form $cX^m$, where 
$m$ lies in $\Z$.

\subsection{The even case: $r=2n$} 
\begin{theorem}\label{locfnaleqn}
Let $F_{v}$ be a $p$-adic field. If $\pi_v$ is an irreducible generic representation of ${\rm GL}_{r}(F_v)$, which occurs as the local constituent of a cuspidal automorphic representation $\pi$ of ${\rm GL}_r$, we have
\[ \dfrac{J(1-s, \rho(w_{n,n})\tilde{W}_{v}, \hat{\phi}_v)}{L(1-s, \tilde{\pi}_v, \wedge^{2})} = \epsilon_{JS}(s, \pi_v, \psi_v, \wedge^{2}) \dfrac{J(s, W_{v}, \phi_v)}{L(s, \pi_v, \wedge^{2})} , \]
where $\epsilon_{JS}(s, \pi_v, \psi_v, \wedge^{2})$ is entire and non-vanishing. 
\end{theorem}
The factor $\epsilon_{JS}(s, \pi_v, \psi_v, \wedge^{2})$ will be defined explicitly in 
\eqref{epsilonjs}.
Note that because of Proposition \ref{arthur-clozel-embed}, square integrable representations satisfy the 
hypotheses of the theorem, and we get an unconditional result in this case.
\begin{proof}
Let
\begin{equation*}
 W =  \prod_{v \in S} W_{v}\cdot \prod_{v \notin S} W_{v}^{0}
\quad\text{and}\quad
 \Phi = \prod_{v \in S} \phi_{v}\cdot \prod_{v \notin S}\phi_{v}^{0} ,
\end{equation*}
where $W_{v}^{0}$ and $\phi_{v}^{0}$ are as in equation \eqref{unrami} and $S=S_{\infty}\cup S_r$, as
before.
From Theorem \ref{global_funteq}, Proposition \ref{globalint_eqlty}, Theorem \ref{global_funteqsh} and equation \eqref{unrami}, we have
\begin{equation} \label{lfqeq1}
 \prod_{v \in S} \dfrac{J(s, W_v, \phi_v)}{L_{Sh}(s, \pi_v, \wedge^{2})} = \dfrac{1}{\epsilon(s, \pi, \wedge^{2})} \cdot \prod_{v \in S} \dfrac{J(1-s, \rho(w_{n,n})\tilde{W}_{v}, \hat{\phi}_v)}{L_{Sh}(1-s, \tilde{\pi}_v, \wedge^{2})}.
\end{equation}

We recall the definition of the function $M_1(q_{v_1}^{-s})$ made in Section \ref{esqfcr} (see \eqref{contra1}). Suppose there are $k$ places in $S_r$. Fix a place $v_1 \in S_r$. There exist $W_{i,v_1}$ and $\phi_{i,v_1}$ such that
\begin{equation} \label{LJS=int}
 L_{JS}(s, \pi_{v_1}, \wedge^2) = \sum_{i=1}^{n_{v_1}} J(s, W_{i,v_1}, \phi_{i,v_1}) 
\end{equation}
and
\begin{equation} \label{LJS=cint}
\sum_{i=1}^{n_{v_1}} J(1-s, \rho(w_{n,n})\tilde{W}_{i,v_1}, \hat{\phi}_{i,v_1}) = M_{1}(q_{v_1}^{-s}) L_{JS}(1-s, \tilde{\pi}_{v_1}, \wedge^2),
\end{equation}
where $ M_{1}(q_{v_1}^{-s})$ is a monomial in $q_{v_1}^{-s}$ (by Lemma \ref{monemon}).
Note that since $\pi$ is cuspidal automorphic, it is globally generic, and hence, every local component 
$\pi_v$ is generic.
By summing equation \eqref{monemon} over $i$ and using Theorem \ref{loceqthm}, we get
\begin{equation} \label{lfqeq2}
 \prod_{v \in S\smallsetminus \{v_1\}} \dfrac{J(s, W_v, \phi_v)}{L_{Sh}(s, \pi_v, \wedge^{2})} = \dfrac{M_{1}(q_{v_1}^{-s})}{\epsilon(s, \pi, \wedge^{2})} \times\prod_{v \in S\smallsetminus \{v_1\}} \dfrac{J(1-s, \rho(w_{n,n})\tilde{W}_{v}, \hat{\phi}_v)}{L_{Sh}(1-s, \tilde{\pi}_v, \wedge^{2})}.
\end{equation}
Fix a place $v_2 \in S\smallsetminus \{v_1\}$. Arguing as above, we have
\begin{equation} \label{lfqeq3}
 \prod_{v \in S\smallsetminus \{v_1, v_2\}} \dfrac{J(s, W_v, \phi_v)}{L_{Sh}(s, \pi_v, \wedge^{2})} = \dfrac{M_{2}(q_{v_2}^{-s})M_{1}(q_{v_1}^{-s})}{\epsilon(s, \pi, \wedge^{2})}\times \prod_{v \in S\smallsetminus \{v_1, v_2\}} \dfrac{J(1-s, \rho(w_{n,n})\tilde{W}_{v}, \hat{\phi}_v)}{L_{Sh}(1-s, \tilde{\pi}_v, \wedge^{2})},
\end{equation}
$M_{2}(q_{v_2}^{-s})$ is a monomial in $q_{v_2}^{-s}$.
Continuing in this way, we get
\begin{align} \label{lfqeq4}
\notag \dfrac{J(s, W_{v_k}, \phi_{v_k})}{L_{Sh}(s, \pi_v, \wedge^{2})} \prod_{v \in S_\infty} \dfrac{J(s, W_v, \phi_v)}{L_{Sh}(s, \pi_v, \wedge^{2})} = 
\dfrac{M'(s)}{\epsilon(s, \pi, \wedge^{2})}  \cdot \dfrac{J(1-s, \rho(w_{n,n})\tilde{W}_{v_k}, \hat{\phi}_{v_k})}{L_{Sh}(1-s, \tilde{\pi}_{v_k}, \wedge^{2})} \\  \times \prod_{v \in S_\infty} \dfrac{J(1-s, \rho(w_{n,n})\tilde{W}_{v}, \hat{\phi}_v)}{L_{Sh}(1-s, \tilde{\pi}_v, \wedge^{2})},
\end{align}
where $ M'(s) = \displaystyle{\prod_{i=1}^{k-1}}M_{i}(q_{v_i}^{-s}).$ Again as above, we get
\begin{equation} \label{lfqeq5}
\prod_{v \in S_\infty} \dfrac{J(s, W_v, \phi_v)}{L_{Sh}(s, \pi_v, \wedge^{2})} =  \dfrac{M(s)}{\epsilon(s, \pi, \wedge^{2})}\cdot \prod_{v \in S_\infty} \dfrac{J(1-s, \rho(w_{n,n})\tilde{W}_{v}, \hat{\phi}_v)}{L_{Sh}(1-s, \tilde{\pi}_v, \wedge^{2})},
\end{equation}
where $ M(s) = \displaystyle{\prod_{i=1}^{k}}M_{i}(q_{v_i}^{-s}).$ From equations \eqref{lfqeq4} and \eqref{lfqeq5}, we have
\begin{equation}\label{lfqeq6}
 \dfrac{J(1-s, \rho(w_{n,n})\tilde{W}_{v_k}, \hat{\phi}_{v_k})}{L_{Sh}(1-s, \tilde{\pi}_{v_{k}}, \wedge^{2})} = M_{k}(q_{v_k}^{-s}) \dfrac{J(s, W_{v_k}, \phi_{v_k})}{L_{Sh}(s, \pi_{v_k}, \wedge^{2})} 
\end{equation}
If we set
\begin{equation}\label{epsilonjs}
 \epsilon_{JS}(s, \pi_{v_k}, \psi_{v_k}, \wedge^{2}) = M_{k}(q_{v_k}^{-s}),
 \end{equation}
 we get
\begin{equation}\label{lfqeq7}
 \dfrac{J(1-s, \rho(w_{n,n})\tilde{W}_{v_k}, \hat{\phi}_{v_k})}{L_{Sh}(1-s, \tilde{\pi}_{v_{k}}, \wedge^{2})} = \epsilon_{JS}(s, \pi_{v_k}, \psi_{v_k}, \wedge^{2}) \dfrac{J(s, W_{v_k}, \phi_{v_k})}{L_{Sh}(s, \pi_{v_k}, \wedge^{2})} 
\end{equation}
The ordering of the ramified places as $1,2,\ldots, k$ was completely arbitrary. Thus equation \eqref{lfqeq7} holds 
when $k$ is replaced by any $i$, $1\le i\le k$.
 Note that we also get the local functional equation at unramified finite places. Let $v_{ur}$ be an unramified finite place. If we choose
\begin{equation*}
 W = W_{v_{ur}} \cdot\prod_{v \in S} W_{v} \cdot \prod_{v \notin S \cup \{v_{ur}\}} W_{v}^{0}
\end{equation*}
and
\begin{equation*}
 \Phi =\phi_{v_{ur}}\cdot \prod_{v \in S} \phi_{v} \cdot \prod_{v \notin S \cup \{ v_{ur}\}}\phi_{v}^{0},
\end{equation*}
and argue as above we get the local functional equation at $v_{ur}$.
\end{proof}
\begin{remark} It may appear {\it a priori} that the function $M_1(q_v^{-s})$ depends on the choices
of vectors $W_{i,v_1}$ and functions $\phi_{i,v_1}$ made in \eqref{LJS=cint} in order to obtain
$\ljsepiv$. However, once we obtain the local functional equation, 
$M_1(q_v^{-s})$ is defined completely by this equation, and is thus independent of this choice. 
The $\epsilon$-factor we have defined thus depends only on the representation $\pi_{v_1}$.
\end{remark}
\begin{remark} If $ F = \Q$ then there is only one infinite place. Hence, from equation \eqref{lfqeq5} we get 
the local functional equation at the archimedean place as well. 
\end{remark}
\subsection{The odd case: $r= 2n +1$}s
\begin{theorem}
Let $F_{v_0}$ be a $p$-adic field. If $\pi_{v_0}$ is an irreducible square integrable representation of ${\rm GL}_{r}(F_{v_0})$, we have
\[ \dfrac{J(1-s, W'_{v_0})}{L(1-s, \tilde{\pi}_{v_0}, \wedge^{2})} = \epsilon_{JS}(s, \pi_{v_0}, \psi_{v_0}, \wedge^{2}) \dfrac{J(s, W_{v_0})}{L(s, \pi_{v_0}, \wedge^{2})} , \]
where $W'_{v_0}$ is some fixed translate of $\tilde{W}_{v_0}$, as in Theorem \ref{global_funteq}, and the function $\epsilon_{JS}(s, \pi_{v_0}, \psi_{v_0}, \wedge^{2})$ is entire and non-vanishing. 
\end{theorem}
\begin{proof}
By Proposition \ref{arthur-clozel-embed}, there exist a cuspidal automorphic representation $\Pi$ of ${\rm GL}_{2n+1}(\af)$ such that $\Pi_{v_0} \backsimeq \pi_{v_0} $. For a place $v_0$, there exist $W_{i,v_0}$ such that
\begin{equation*} 
 L_{JS}(s, \pi_{v_0}, \wedge^2) = \sum_{i=1}^{n_{v_0}} J(s, W_{i,v_0}) 
\end{equation*}
and
\begin{equation*} 
\sum_{i=1}^{n_{v_0}} J(1-s, W'_{i,v_0}) = M_{0}(q_{v_0}^{-s}) L_{JS}(1-s, \tilde{\pi}_{v_0}, \wedge^2),
\end{equation*}
where $M_{0}(q_{v_0}^{-s})$ is a monomial in $q_{v_0}^{-s}$ (by Lemma \ref{monemon}).
As in the even case, we have 
\begin{equation} \label{lfqeq-odd1}
 \prod_{v \in S} \dfrac{J(s, W_v)}{L_{Sh}(s, \Pi_v, \wedge^{2})} = \dfrac{1}{\epsilon(s, \pi, \wedge^{2})} \cdot \prod_{v \in S} \dfrac{J(1-s, W'_{v})}{L_{Sh}(1-s, \tilde{\Pi}_v, \wedge^{2})}.
\end{equation}
Again arguing as in the even case for the place $v_0$, and using Theorem \ref{sqinteqthm}, we have
\begin{equation} \label{lfqeq-odd2}
 \prod_{v \in S\smallsetminus \{v_0\}} \dfrac{J(s, W_v)}{L_{Sh}(s, \Pi_v, \wedge^{2})} = \dfrac{M_{0}(q_{v_0}^{-s})}{\epsilon(s, \Pi, \wedge^{2})} \cdot \prod_{v \in S\smallsetminus \{v_0\}} \dfrac{J(1-s, W'_{v})}{L_{Sh}(1-s, \tilde{\Pi}_v, \wedge^{2})}
\end{equation}
From equations \eqref{lfqeq-odd1} and \eqref{lfqeq-odd2}, we have
 \begin{equation}
 \dfrac{J(1-s, W'_{v_0})}{L(1-s, \tilde{\pi}_{v_0}, \wedge^{2})} = M_{0}(q_{v_0}^{-s})( \dfrac{J(s, W_{v_0})}{L(s, \pi_{v_0}, \wedge^{2})}.
\end{equation}
Set $\epsilon_{JS}(s, \pi_{v_0}, \psi_{v_0}, \wedge^{2}) = M_{0}(q_{v_0}^{-s})$. We get
\[
 \dfrac{J(1-s, W'_{v_0})}{L(1-s, \tilde{\pi}_{v_0}, \wedge^{2})} = \epsilon_{JS}(s, \pi_{v_0}, \psi_{v_0}, \wedge^{2}) \dfrac{J(s, W_{v_0})}{L(s, \pi_{v_0}, \wedge^{2})}.
\]
\end{proof}



\begin{thebibliography}{10}

\bibitem[AC89]{ArCl89}
James Arthur and Laurent Clozel.
\newblock {\em Simple algebras, base change, and the advanced theory of the
  trace formula}, volume 120 of {\em Annals of Mathematics Studies}.
\newblock Princeton University Press, Princeton, NJ, 1989.

\bibitem[AR05]{AnRa05}
U.~K. Anandavardhanan and C.~S. Rajan.
\newblock Distinguished representations, base change, and reducibility for
  unitary groups.
\newblock {\em Int. Math. Res. Not.}, (14):841--854, 2005.

\bibitem[Bel11]{Belt11}
Dustin Belt.
\newblock On the {H}olomorphy of {E}xterior-{S}quare {$L$}-functions for {${\rm
  GL}(n)$}.
\newblock {\em preprint}, 2011.

\bibitem[BZ77]{BeZe77}
I.~N. Bernstein and A.~V. Zelevinsky.
\newblock Induced representations of reductive $p$-adic groups. {I}.
\newblock {\em Ann. Sci. \'Ecole Norm. Sup. (4)}, 10(4):441--472, 1977.

\bibitem[CPS94]{CoPS942}
James~W. Cogdell and Ilya~I. Piatetski-Shapiro.
\newblock Exterior square {$L$}-function for {${\rm GL}(n)$}.
\newblock {\em preprint}, 1994.

\bibitem[GS01]{GeSh01}
Stephen Gelbart and Freydoon Shahidi.
\newblock Boundedness of automorphic {$L$}-functions in vertical strips.
\newblock {\em J. Amer. Math. Soc.}, 14(1):79--107 (electronic), 2001.

\bibitem[Hen00]{Henniart00}
Guy Henniart.
\newblock Une preuve simple des conjectures de {L}anglands pour {${\rm GL}(n)$}
  sur un corps {$p$}-adique.
\newblock {\em Invent. Math.}, 139(2):439--455, 2000.

\bibitem[Hen10]{Henniart10}
Guy Henniart.
\newblock Correspondance de {L}anglands et fonctions {$L$} des carr\'es
  ext\'erieur et sym\'etrique.
\newblock {\em Int. Math. Res. Not. IMRN}, (4):633--673, 2010.

\bibitem[HT01]{HaTa01}
Michael Harris and Richard Taylor.
\newblock {\em The geometry and cohomology of some simple {S}himura varieties},
  volume 151 of {\em Annals of Mathematics Studies}.
\newblock Princeton University Press, Princeton, NJ, 2001.
\newblock With an appendix by Vladimir G. Berkovich.

\bibitem[JS81]{JASh811}
H.~Jacquet and J.~A. Shalika.
\newblock On {E}uler products and the classification of automorphic
  representations. {I}.
\newblock {\em Amer. J. Math.}, 103(3):499--558, 1981.

\bibitem[JS90]{JaSh90a}
Herv{\'e} Jacquet and Joseph Shalika.
\newblock Exterior square {$L$}-functions.
\newblock In {\em Automorphic forms, {S}himura varieties, and {$L$}-functions,
  {V}ol.\ {II} ({A}nn {A}rbor, {MI}, 1988)}, volume~11 of {\em Perspect.
  Math.}, pages 143--226. Academic Press, Boston, MA, 1990.

\bibitem[Kab04]{Kable04}
Anthony~C. Kable.
\newblock Asai {$L$}-functions and {J}acquet's conjecture.
\newblock {\em Amer. J. Math.}, 126(4):789--820, 2004.

\bibitem[Kew11]{Kewat11}
Pramod~Kumar Kewat.
\newblock The local exterior square {$L$}-function: holomorphy, non-vanishing
  and {S}halika functionals.
\newblock {\em preprint (to appear in J. Algebra)}, 2011.

\bibitem[Kim99]{Kim99}
Henry~H. Kim.
\newblock Langlands-{S}hahidi method and poles of automorphic {$L$}-functions:
  application to exterior square {$L$}-functions.
\newblock {\em Canad. J. Math.}, 51(4):835--849, 1999.

\bibitem[Sha81]{Shahidi81}
Freydoon Shahidi.
\newblock On certain {$L$}-functions.
\newblock {\em Amer. J. Math.}, 103(2):297--355, 1981.

\bibitem[Sha85]{Shahidi85}
Freydoon Shahidi.
\newblock Local coefficients as {A}rtin factors for real groups.
\newblock {\em Duke Math. J.}, 52(4):973--1007, 1985.

\bibitem[Sha90]{Shahidi90}
Freydoon Shahidi.
\newblock A proof of {L}anglands' conjecture on {P}lancherel measures;
  complementary series for {$p$}-adic groups.
\newblock {\em Ann. of Math. (2)}, 132(2):273--330, 1990.



\end{thebibliography}

\end{document}